\documentclass[a4paper,11pt]{amsart}
\usepackage{amsmath,amssymb,amsthm} 
\usepackage{graphics} 

\usepackage{epsfig}
\usepackage{color}

\numberwithin{equation}{section}
\newtheorem{theorem}{Theorem}[section]
\newtheorem{lemma}[theorem]{Lemma}
\newtheorem{proposition}[theorem]{Proposition}

\theoremstyle{definition}

\newtheorem{remark}[theorem]{Remark}

\newcommand{\R}{\mathbb{R}}
\newcommand{\N}{\mathbf{N}}

\newcommand{\eps}{\varepsilon}

\newcommand{\be}{\begin{equation}}
\newcommand{\ee}{\end{equation}}

\newcommand\lt{\left}
\newcommand\rt{\right}

\def\E{\mathcal{E}}

\def\diam{{\rm diam}}
\def\spt{{\rm spt}}

\def\cWp{\mathcal{W}_p}
\newcommand{\cWpa}[1]{\lt[\cWp^p(#1)\rt]^\alpha}

\title[Isoperimetric problem with Wasserstein interaction]{Existence and stability results for an isoperimetric problem with a non-local interaction of Wasserstein type}
\date{\today}
\author[J. Candau-Tilh]{Jules Candau-Tilh}
\address{J.C.: Ecole Normale Sup\'erieure, ENS, F-75005 Paris \& Universit\'e de Paris, Sorbonne Universit\'e, Laboratoire Jacques-Louis Lions, LJLL, F-75013 Paris}
\email{candau@clipper.ens.psl.eu}

\author[M. Goldman]{Michael Goldman}
\address{M.G.: Universit\'e de Paris, Sorbonne Universit\'e,  CNRS,  Laboratoire Jacques-Louis Lions, LJLL, F-75013 Paris}
\email{goldman@math.univ-paris-diderot.fr}

\begin{document}
\begin{abstract} 
The aim of this paper is to prove the existence of minimizers for a variational problem involving the minimization under volume constraint of the sum of the perimeter and a non-local energy of Wasserstein type.
This extends previous partial results to the full range of parameters. We also show that in the regime where the perimeter is dominant, the energy is uniquely minimized by balls.
\end{abstract}

\maketitle

\section{Introduction}
In this paper we consider a variational problem first proposed in \cite{PeRo} as a  model describing  the formation of bi-layers cellular  membranes. Our first main result 
is the proof of the existence of minimizers in every space dimension and for every value of the parameters in the model. This extends previous results obtained in \cite{BuCaLa,XiaZhou} to which we refer for further motivation of the problem. 
Our second main result is a proof of the minimality of the ball in the regime where the perimeter is dominant. To be more concrete, denoting by $W_p$ 
the Wasserstein distance for $p\ge 1$ (see \cite{Viltop}) and identifying a set $E\subset \R^d$ with the restriction of the Lebesgue measure to $E$, we introduce the non-local energy
\begin{equation}\label{def:Wp}
 \cWp(E)=\inf_{|F\cap E|=0} W_p(E,F).
\end{equation}
As already noticed in \cite{BuCaLa}, this may be viewed as a projection problem for the Wasserstein distance (see \cite{DePMSV}). We then consider for $\lambda, \alpha>0$ the variational problem
\begin{equation}\label{def:minprob}
 \inf_{|E|=\omega_d} P(E)+ \lambda \cWpa{E},
\end{equation}
where $\omega_d$ is the volume of the unit ball and  $P(E)$ denotes the perimeter of $E$, see \cite{Maggi}. Let us point out that probably the two most interesting cases are $\alpha=1$ and $\alpha= \frac{1}{p}$. Our first main result is the following:
\begin{theorem}\label{exis_mini}
 For every $d\ge 2$, $p\ge 1$, $\alpha>0$ and $\lambda>0$, problem \eqref{def:minprob} has minimizers.
 Moreover, there exists $C=C(d,p,\alpha)>0$ such that if $E=\cup_{i=1}^I E^i$ is such a minimizer with $E^i$ the connected components of $E$, then 
 \[
  \sum_{i=1}^I \diam(E^i)\le C (1+\lambda)^{\frac{(d-1)(1+p)}{1+\alpha p}} \quad \textrm{ and } \quad \inf_i \diam(E^i)\ge C  (1+\lambda)^{-\frac{1+p}{1+\alpha p}}.
 \] 
 As a consequence $I\le C (1+\lambda)^{\frac{d(1+p)}{1+\alpha p}}$.
\end{theorem}
Notice that we can actually say much more about the regularity of the minimizers, see Remark \ref{rem:reg}. This result was first obtained in the case $d=2$ in \cite{BuCaLa} and then extended to the case $d\ge 3$ in \cite{XiaZhou} but under  the assumption that $\lambda$ is small together
with some  restrictions on $\alpha$. The idea of the proof, which is by now well-established in the context of geometrical variational problems  (see e.g. \cite{golnov, KnMuNov, FraLieb, novprat}), 
is to follow a concentration-compactness type argument. 
We first show that thanks to the isoperimetric inequality,
lack of compactness for minimizing sequences can only come from splitting of the mass.  This leads to the existence of so-called generalized minimizers (see Proposition \ref{prop:existgenmin}).
Then, we show following \cite{BuCaLa}, that these generalized minimizers are actually $\Lambda-$minimizers of the perimeter (see \cite{Maggi}) and therefore have uniform density bounds. 
As a direct consequence, we obtain that they are made of a finite number of uniformly bounded connected components. At this point the proof of the existence is concluded as in \cite{BuCaLa}
using the fact that the non-local energy $\cWp^p$ is additive for sets which are sufficiently far apart.  \\

Our second main result is that if $\lambda$ is small enough then \eqref{def:minprob} is uniquely minimized by balls.
\begin{theorem}\label{ballmin}
 For every $d\ge 2$, $p\ge 1$ and $\alpha>0$, there exists $\lambda_0>0$ such that for every $\lambda\le \lambda_0$, balls are the only minimizers of \eqref{def:minprob}.
\end{theorem}
\begin{remark}
 Let us point out that if we considered the volume as the relevant parameter and replaced \eqref{def:minprob} by
 \[
  \min_{|E|=m} P(E)+\cWpa{E},
 \]
then by scaling (see \cite{XiaZhou}) we would obtain that balls are the unique minimizers for small $m$ if $ \alpha\lt(1+\frac{p}{d}\rt)+\frac{1}{d}>1$
(which is essentially the case for which \cite{XiaZhou} obtained the existence of minimizers) while balls are the unique minimizers for large $m$ if $  \alpha\lt(1+\frac{p}{d}\rt)+\frac{1}{d}<1$.
\end{remark}

Again, this result is neither surprising by its statement nor by the strategy to prove it. Indeed, following the pioneering work of Cicalese and Leonardi
which gave in \cite{CicLeo} an alternative proof of the quantitative isoperimetric inequality, it has been understood that such stability results may be obtained by combining the
regularity theory for $\Lambda-$minimizers of the perimeter together with a (usually delicate) Taylor expansion of the energy around the ball. 
This second part of the proof is often referred to as a Fuglede type argument, see \cite{fuglede}. 
Let us cite \cite{KnMu, AcFuMo, F2M3,CaFuPra, MuVes} as a few examples where this strategy has been carried out.   The main difficulty here is that our non-local energy depends
in a very implicit way on the competitor. Moreover, as opposed to \cite{AcFuMo,MuVes, GolNovRuf}, the underlying PDE is non-linear (namely the Monge-Amp\`ere equation)
making it very difficult to use standard tools from shape optimization such as shape derivatives. This makes the exact computation of the Taylor expansion of the energy challenging.
We go around this difficulty by plugging in the dual formulation of optimal transport the Kantorovich potentials corresponding to the ball and show that this leads to a lower bound which is good enough for our purpose (see Proposition \ref{prop:fugledeW}).
\\

\textbf{Related results in the literature.} In the footsteps of \cite{KnMu} there has been an intense research activity around isoperimetric problems with non-local interactions. 
Probably the simplest and most studied one is the Gamow liquid-drop model where the non-local part of the energy is given by a Riesz type interaction energy. For this model, it has been shown that
generalized minimizers exist and are balls for small volume (see \cite{KnMu, F2M3,CaFuPra, novprat} and the review paper \cite{ChMuTo}). However, as opposed to our setting, 
it has been proven for the liquid-drop model that under some restrictions on the parameters, classical minimizers do not exist for large volumes (see \cite{KnMu,FranNam}).
This is due to the long-range nature of the interactions induced by the Riesz kernel (in comparison with Proposition \ref{prop:super}). 
Indeed, for compactly supported kernels it is shown in \cite{Rigot} that  minimizers exist for all volumes  (see also \cite{Pegon}).


\subsection*{Acknowledgements} This work was partially supported by the ANR project SHAPO.\\

Shortly before submitting this paper, Novack, Topaloglu and Venkatraman proved in \cite{NoToVenk} (uploaded on the Arxiv the 10th of August 2021) essentially the same existence 
result as Theorem \ref{exis_mini}. While the basic ingredients of the proofs are similar
(a combination of the isoperimetric inequality to avoid the loss of mass for minimizing sequences together with a quasi-minimality property in order to obtain  density estimates), 
the implementation is quite different. Indeed, in the present work we avoid the use of Almgren nucleation Lemma and rely simply on the relative isoperimetric inequality.
Moreover, we separate the compactness and the regularity issues with the use of generalized minimizers. As a result, we can directly rely on the well-established regularity
theory for $\Lambda-$minimizers of the perimeter. Finally, our result is more quantitative thanks to a more explicit treatment of the volume constraint (see Proposition \ref{prop:relax})
and our interpolation inequality (see Proposition \ref{prop:interpol}).

\section{The non-local energy}\label{sec:nonlocal}
In this section we gather a few useful results about the energy $\cWp$ defined in \eqref{def:Wp}. Most of these results were obtained in the case of bounded sets in \cite{BuCaLa,XiaZhou} but often with quite different proofs. We start with the well-posedness of \eqref{def:Wp}.
\begin{proposition}\label{prop:Wb}
 There exists $C=C(d,p)>0$ such that for every set  $E\subset \R^d$,
 \begin{equation}\label{nrj_bdd}
  \cWp(E)\le C |E|^{\frac{1}{p}+\frac{1}{d}}.
 \end{equation}
Moreover, if $|E|<\infty$, the minimization problem \eqref{def:Wp} is attained by a unique minimizer $F$ and if $\pi$ is an optimal transport plan\footnote{for $p>1$ we know from \cite[Theorem 2.44]{Viltop} 
that $\pi$ is unique and is  induced by a map but for $p=1$, since we do not assume finite moments for $E$ it does not follow from  \cite[Theorem 2.50]{Viltop}.}for $W_p(E,F)$,  we have the estimate
\begin{equation}\label{eq:Linf}
  |x-y|\le C |E|^{\frac{1}{d}} \qquad \textrm{for } \pi-a.e. \ (x,y).
\end{equation}

\end{proposition}

\begin{proof}
We may assume without loss of generality that $|E|<\infty$ otherwise there is nothing to prove. By scaling we can further assume that $|E|=1$.
In order to prove \eqref{nrj_bdd}, we will construct a partition $(E_i)_{i \geq 1}$ of $E$ such that each $E_i$ can be transported with a well-controlled cost.
To this aim, consider a partition of $\R^d$ into cubes $(Q_i)_{i \geq 1}$ of sidelength $\ell = 2^{1/d}$. Since $|E|=1$,   if we define $E_i=E\cap Q_i$ we have $|E_i|\le |Q_i|/2$ for every $i$. 
 Therefore we can find  a set $F_i \subset Q_i$ such that $|E \cap F_i| = 0$ and $|F_i| = |E_i|$. If  $T_i$ is the optimal transport map (in fact any transport map would work)
 from $E_i$ to $F_i$ we have 
	\[
	\sup_{E_i} |T_i - x| \leq C.
	\] 
	\noindent 
	Finally, consider  $F=\cup_i F_i$ and $T$   the map whose restriction to each $Q_i$ is $T_i$. The map $T$ is a transport map from $E$ to $F$ and 
	\[
	\mathcal{W}_p(E) \leq W_p(E, F) \leq \bigg(\sum_{i \geq 1} \int_{E_i} |T_i-x|^p \, \bigg)^{\frac{1}{p}} \leq C.
	\]
	\noindent
	This proves \eqref{nrj_bdd}.
	\bigskip
	
	Existence and uniqueness of a minimizer $F$ for \eqref{def:Wp} follows from 
	\cite[Proposition 5.2]{DePMSV} (which is stated for $p = 2$, but generalizes easily to any $p \geq 1$) with $f=\chi_{E^c}$ and $\Omega=\R^d$.
	Moreover, as a consequence of \cite[Proposition 5.2]{DePMSV} we have 
	\begin{equation}\label{eq:equalrelaxW}
  	\widetilde{\cWp}(E)=\inf_{\mu}\lt\{ W_p(E,\mu) \ : \mu\le \chi_{E^c}\rt\} =\cWp(E).
 	\end{equation}
	and $\chi_F$ is also the unique minimizer of $\widetilde{\cWp}(E)$.
	Let $\pi$ be an optimal transport plan for $W_p(E,F)$ and let us show \eqref{eq:Linf}. For this we 
	 adapt the proof of \cite[Lemma 4.3]{XiaZhou} to the case of plans instead of maps. Letting
\[
 \Gamma=\{(x,y)\in \spt \pi \ : \ |x-y|\ge C\}
\]
let us show that for $C$ large enough, $\pi(\Gamma)=0$. Assume that it is not the case and let $R$ be such that $|B_{R}|=3$. 
Then, there exists $x\in \R^d$ such that $m=\pi(\Gamma\cap (B_{R}(x)\times \R^d))>0$. Without loss of generality we may assume that $x=0$.
Let $\pi_{\rm bad}=\chi_{\Gamma\cap (B_{R}\times \R^d)}\pi$.
Since $|B_{R}|-|E|-|F|\ge 1\ge m>0$, 
there exists $\widetilde{\mu}\le \chi_{B_{R}}(1-\chi_E-\chi_F)$ with $\widetilde{\mu}(\R^d)=\pi_{\rm bad}(\R^d\times \R^d)$. 
Finally let $\theta$ be the first marginal of $\pi_{\rm bad}$ and set 
\[
 \tilde{\pi}= \pi-\pi_{\rm bad} +\frac{1}{m} \theta\otimes \tilde \mu.
\]
It is readily checked that the first marginal of $\tilde{\pi}$ is $\chi_E$ and that its second marginal $\mu$ satisfies $\mu\le \chi_{E^c}$. We thus have on the one hand by definition
of $\Gamma$
\[
 W^p_p(E,F)\ge \int_{(\Gamma\cap(B_{R}\times \R^d))^c} |x-y|^p d\pi + m C^p.
\]
On the other hand, by minimality of $F$ for $\widetilde{W}_p(E)$,
\[
 W_p^p(E,F)\le W_p^p(E,\mu)\le \int_{(\Gamma\cap(B_{R}\times \R^d))^c} |x-y|^p d\pi + m2^p R^p.
\]
This implies $C<2 R$ and concludes the proof that $\pi(\Gamma)=0$ if $C$ is large enough.

\end{proof}
We now turn to the super-additivity and lower semi-continuity of $\cWp$.
\begin{proposition}\label{prop:super}
 We have:
 \begin{itemize}
  \item[(i)] If $E$ and $E'$ are disjoint sets  then 
  \begin{equation}\label{repulsive}
   \cWp^p(E\cup E')\ge  \cWp^p(E)+ \cWp^p( E').
  \end{equation}
As a consequence, if $E\subset E'$ then $\cWp(E)\le \cWp(E')$;
\item[(ii)] There exists $C>0$ such that if 
\[d(E,E')\ge C \max(|E|^{\frac{1}{d}}, |E'|^{\frac{1}{d}}),\]
then
\[
 \cWp^p(E\cup E')=  \cWp^p(E)+ \cWp^p( E');
\]
\item[(iii)] If $E_n$ converges in $L^1_{loc}$ to $E$ then 
\begin{equation}\label{semicont}
		\mathcal{W}_p(E) \leq \liminf_n \mathcal{W}_p(E_n).
		\end{equation}
 \end{itemize}

\end{proposition}
\begin{proof}
To prove $(i)$, let $F$ be  the $\mathcal{W}_p$-minimizer for $E \cup E'$, and $\pi$ be  an optimal transport plan from $E \cup E'$ to $F$. 
Let $\mu_E$ be the second marginal of $\chi_{E\times\R^d}\pi$ and $\mu_{E'}$
be the second marginal of $\chi_{E'\times\R^d}\pi$. By definition $\mu_E$ is   $\widetilde{\cWp}$-admissible (recall \eqref{eq:equalrelaxW}) for $E$.
Moreover, $\chi_{E\times\R^d}\pi$ is an optimal transport plan between $E$ and $\mu_E$. The corresponding statement also holds for $E'$ instead of $E$. 
Therefore, appealing once more to \eqref{eq:equalrelaxW},
	\[
 	\mathcal{W}^p_p(E) + \mathcal{W}^p_p(E') \leq W_p^p(E, \mu_E) + W_p^p(E', \mu_{E'}) = \int_{(E \cup E')\times \R^d} |x-y|^p \, \mathrm{d}\pi = \mathcal{W}_p^p(E \cup E').
 	\]
 	
 	Property $(ii)$ is a direct consequence of \eqref{eq:Linf}. Indeed, if $F$ and $F'$ are the $\mathcal{W}_p$-minimizers for $E$ and  $ E'$, by \eqref{eq:Linf}, $|F\cap F'|=0$ so that $F\cup F'$ 
 	is admissible for $E\cup E'$ which gives $\cWp(E\cup E')\le \cWp(E)+\cWp(E')$.
 	\bigskip
 	
 	We finally prove $(iii)$, and  consider a sequence $(E_n)_{n \geq 1}$ that is $L^1_{\mathrm{loc}}$ converging to $E$. 
 	For every $R>0$ set $E_{R,n}=E_n\cap B_R$ so that $E_{R,n}$ converges in $L^1$ to $E_R=E\cap B_R$. Using the continuity of $W_p$ with respect to weak convergence,
 	\eqref{eq:Linf} and \eqref{eq:equalrelaxW} it is not hard to check  that $\cWp$ is lower semi-continuous with respect to $L^1$ convergence (in Lemma \ref{wassloc} below we will actually prove a much stronger result).
 	Since by \eqref{repulsive}, $\cWp(E_{R,n})\le \cWp(E_n)$ we have 
 	\[
 	 \cWp(E_R)\le \liminf_{n\to \infty} \cWp(E_{R,n})\le \liminf_{n\to \infty} \cWp(E_n).
 	\]
Since $E_R$ converges in $L^1$ to $E$ as $R\to \infty$, using once more the lower semi-continuity of $\cWp$ for this convergence we conclude the proof.
\end{proof}
\begin{remark}\label{rem:contR}
 Let us point out that for every set $E$ with $|E|<\infty$, since $E\cap B_R$ converges in $L^1$ to $E$ as $R\to \infty$, we have by lower semi-continuity 
 and $\cWp(E\cap B_R)\le \cWp(E)$ that $\lim_{R\to \infty} \cWp(E\cap B_R)=\cWp(E)$.
\end{remark}

We then prove that $\cWp^p$ is  Lipschitz continuous with respect to $L^1$ convergence.   This is a  crucial ingredient in order to obtain the $\Lambda-$minimality property of generalized minimizers.
\begin{lemma}\label{wassloc}
	There exists a constant $C = C(d,p)>0$ such that for any Lebesgue sets $E,E'$
	
	\begin{equation}\label{eq:LipW}
	| \mathcal{W}_p^p(E) - \mathcal{W}_p^p(E') |\leq C (|E|^{\frac{p}{d}}+|E'|^{\frac{p}{d}})|E \Delta E'|.
	\end{equation}
	Moreover, there exists $C=C(d,p,\alpha)>0$ such that for every family of sets $(E^i)_{i\ge 1}$ and $((E')^i)_{i\ge 1}$,
	\begin{multline}\label{eq:LipW2}
 \lt| \lt[\sum_i \cWp^p(E^i)\rt]^\alpha-\lt[\sum_{i} \cWp^p((E')^i)\rt]^\alpha\rt|\\
 \le C \max\lt(\lt(\sum_i \cWp^p(E^i)\rt)^{\alpha-1}, \lt(\sum_i \cWp^p((E')^i)\rt)^{\alpha-1}\rt) \lt|\sum_i \cWp^p(E^i)-\sum_i \cWp^p((E')^i)\rt|.
\end{multline}
\end{lemma}
\begin{proof}
We start with the proof of  \eqref{eq:LipW}. Thanks to Remark \ref{rem:contR} we may assume that $E$ and $E'$ are bounded sets. By symmetry of the roles of $E$ and $E'$, it is sufficient to show that 
	\begin{equation}\label{eq:toproveLip}
	\mathcal{W}_p^p(E') - \mathcal{W}^p_p(E) \leq C|E'|^{\frac{p}{d}} |E' \setminus E|.
	\end{equation}
	\noindent By scaling we may assume that $|E'|=1$. Let $F$ with $|E\cap F|=0$ be such that $\mathcal{W}^p(E) = W_p^p(E, F)$.
	Let $T_E$ be an optimal transport map from $E$ to $F$  (which exists   by \cite[Theorem 2.44 \& Theorem 2.50]{Viltop} since $E$ and $F$ are bounded), and denote $T_F=T_E^{-1}$ which is an optimal transport map from $F$ to $E$.
	We define $\widetilde{F} = F \setminus E'$, set $F^-= T_E(E')\cap \widetilde{F}$ and decompose $E'$ as 
	\[
	E' = (E'\cap T_F(\widetilde{F})) \cup (E' \setminus T_F(\widetilde{F}))
	\]
	so that $T_E(E'\cap T_F(\widetilde{F}))=F^-$.
	Our goal is now to construct a set $F^+\subset (E'\cup F^{-})^c$ and a map $T^+$ from $E' \setminus T_F(\widetilde{F})$ to $F^+$ with controlled transport cost.   
	We proceed as in the proof of \eqref{nrj_bdd} and consider a partition of $\R^d$ into cubes $(Q_i)_{i \geq 1}$ of sidelength $\ell = 3^{1/d}$. We thus have for every 
	$i \geq 1$, 
	\[ |Q_i|-|E'\cap Q_i|-|F^{-}\cap Q_i| \geq |E' \setminus T_F(\widetilde{F})|.\]
	Therefore, for any $i \geq 1$, there exists $F_i\subset Q_i\cap (E'\cup F^{-})^c$ such that  $|F_i| = | (E' \setminus T_F(\widetilde{F}))\cap Q_i|$ and an optimal transport map $T_i$ from $(E' \setminus T_F(\widetilde{F}))\cap Q_i$ to $F_i$. 
	We set $F^+ = \textstyle\cup_{i \geq 1} F_i$ and define  $T^+$ from $E' \setminus T_F(\widetilde{F})$ to $F^+$ by setting its restriction on any $Q_i$ to be $T_i$. By construction,
	\[
	 \sup_{E' \setminus T_F(\widetilde{F})} |T^+-x|\le C.
	\]
	We can now set $T= T_E$ on $E'\cap T_F(\widetilde{F})$ and $T=T^+$ on $E' \setminus T_F(\widetilde{F})$ and obtain
	\[
	\begin{aligned}
	\mathcal{W}_p^p(E') - \mathcal{W}_p^p(E) &\leq \int_{E'\cap T_F(\widetilde{F})}| T_E-x|^p  + \int_{E' \setminus T_F(\widetilde{F})}|T^+-x|^p - \mathcal{W}_p^p(E) \\
	&\leq \int_{E' \setminus T_F(\widetilde{F})}|T^+-x|^p\\
	&\leq C |E' \setminus  T_F(\widetilde{F})|.
	\end{aligned}
	\]
	\noindent
	We finally observe that 
	\[
	\begin{aligned}
	|E' \setminus  T_F(\widetilde{F})|&\leq |E' \setminus E| + |E \backslash T_F(\widetilde{F})|\\
	&= |E' \setminus E| + |E| - |F\backslash E' | \\
	&\leq |E'\setminus E| +  |E' \cap F| \\
	&\leq 2 |E' \setminus E|.\\
	\end{aligned}
	\]
	\noindent This proves \eqref{eq:toproveLip}.
	\bigskip

 We now turn to \eqref{eq:LipW2}. For this we simply use the fact that there exists $C=C(\alpha)>0$ such that for every $a>0$ and $b>0$
 \[
  | a^\alpha-b^\alpha|\le C \max(a^{\alpha-1},b^{\alpha-1}) |a-b|. 
 \]
\end{proof}
From \eqref{eq:LipW2}, we see that  in order to obtain a good Lipschitz bound for  $E\mapsto\cWpa{E}$ when $\alpha< 1$ (recall that we are particularly interested in the case $\alpha=\frac{1}{p}\le 1$), we will need a control from below on the transport term. 
This is obtained through the following interpolation result between the perimeter and $\cWp$. 
\begin{proposition}\label{prop:interpol}
 There exists a constant $C=C(d)>0$ such that for every family of sets $(E^i)_{i\ge 1}$ we have 
 \begin{equation}\label{eq:interpol}
  \lt(\sum_i \cWp^p(E^i)\rt)^{\frac{1}{p}}\lt(\sum_i P(E^i)\rt) \ge C \lt(\sum_i |E^i|\rt)^{1+\frac{1}{p}}
 \end{equation}

\end{proposition}
\begin{proof}
Since by H\"older inequality, $\mathcal{W}_1(E)\le \cWp(E)|E|^{1-\frac{1}{p}}$, using H\"older inequality once more for the sum we see that it is enough to prove \eqref{eq:interpol} for $p=1$. Let $E$ be a set of finite perimeter and volume. We will first show that there exists  $C=C(d)>0$ such that
	\begin{equation}\label{eq:lowbound}
	\mathcal{W}_1(E) \geq C r(|E| - Cr P(E)).
	\end{equation}
	\noindent Take $\eta$ a standard mollifier, rescale it by setting $\eta_r(x) = r^{-d}\eta(x/r)$ and consider $\phi_r = \eta_r \ast \chi_E$. Using Young's inequality,
	we have
	\[
	|\nabla \phi_r |_{\infty} \leq |\chi_E|_{\infty} |\nabla \eta_r|_1 \leq C r^{-1}.
	\]
	\noindent Therefore, by Kantorovich duality for $W_1$, we obtain using $F\subset E^c$,
	\[
	\mathcal{W}_1(E) = W_1(E, F)= \sup_{|\nabla \psi|\leq 1} \int \psi (\chi_E - \chi_{F})
	\geq C \int r \phi_r (\chi_E - \chi_{F})
	\geq C r \int \phi_r (\chi_E - \chi_{E^c}).
	\]
	\noindent Since $\int \phi_r=|E|$,
	\[	
	\int \phi_r \chi_E = |E| - \int \phi_r (1-\chi_E)=|E|-\int \phi_r \chi_{E^c},
	\]
	\noindent so that 
	\[
	\mathcal{W}_1 (E) \geq C r \bigg(|E| - 2 \int \phi_r \chi_{E^c}\bigg).
	\]
	\noindent We now re-express the term $\textstyle{\int} \phi_r \chi_{E^c}$ in order to bound it by the perimeter of $E$ : 	
	\[
	\begin{aligned}
	\int \phi_r \chi_{E^c} &= \iint \eta_r(y-x) \chi_E(x) \chi_{E^c}(y) \, \mathrm{d}x\mathrm{d}y \\
	&=\frac{1}{2} \iint \eta_r(x-y)|\chi_E(x) - \chi_{E}(y)|\, \mathrm{d}x\mathrm{d}y \\
	&= \frac{1}{2} \iint \eta_r (z) | \chi_E(x) - \chi_E(x+z)| \, \mathrm{d}x\mathrm{d}z \\
	&\leq C P(E) \int |z| \eta_r (z) \, \mathrm{d}z\\
	&\leq C r P(E).
	\end{aligned}
	\]
	This proves \eqref{eq:lowbound}.
	\noindent Let now $(E^i)_{i\ge 1}$ be a family of sets and let us show \eqref{eq:interpol}. We may assume that $\sum_i P(E^i)+|E^i|<\infty$ since otherwise there is nothing to prove. 
	Summing \eqref{eq:lowbound} over $i \geq 1$ yields
	\[
	\sum_{i \geq 1}\mathcal{W}_1(E^i) \geq C r \bigg(\sum_{i \geq 1} |E^i| - C r \sum_{i \geq 1} P(E^i) \bigg).
	\]
	\noindent We conclude the proof by taking
	\[
	r = \eps \frac{\sum_{i \geq 1} |E^i|}{\sum_{i \geq 1} P(E^i)},
	\]
	\noindent with $\eps$ chosen small enough so that $\eps C \leq 1/2$.
 \end{proof}

\section{Existence of minimizers}\label{sec:exis}
In this section we prove Theorem \ref{exis_mini}. As already explained in the introduction, we will first prove the existence of generalized minimizers and then prove that they are $\Lambda-$minimizers
of the perimeter to obtain a bound on their diameter which readily implies the existence of minimizers in a classical sense.
\subsection{Existence of generalized minimizers}
We start with some notation.  For a set $E$ we define the energy (we keep the dependence in $p$ and $\alpha$ implicit)
\[
 \E_\lambda(E)= P(E)+\lambda \cWpa{E}.
\]
We call a family  $\widetilde{E}=(E^i)_{i\ge 1}$ a generalized set and define the generalized energy as 
\begin{equation}\label{def:gener}
 \widetilde{\E}_\lambda(\widetilde{E})=\sum_i P(E^i)+\lambda\lt[\sum_i \cWp^p(E^i)\rt]^\alpha.
\end{equation}
We say that $\widetilde{E}$ is a generalized minimizer if $\sum_i |E^i|=\omega_d$ and 
\[
 \widetilde{\E}_\lambda(\widetilde{E})=\inf\lt\{\widetilde{\E}_\lambda(\widetilde{E}') \ : \ \sum_i |(E')^i|=\omega_d\rt\}.
\]
\begin{proposition}\label{prop:existgenmin}
 For every $d\ge 2$, $p\ge 1$, $\alpha>0$ and $\lambda>0$, there exists generalized minimizers and 
 \begin{equation}\label{eq:equalinf}
  \inf\lt\{ \E_\lambda(E) \ : |E|=\omega_d\rt\}=\inf\lt\{ \widetilde{\E}_\lambda(\widetilde{E}) \ : \ \sum_i |E^i|=\omega_d\rt\}.
 \end{equation}

\end{proposition}
\begin{proof}
 We start by pointing out that using Proposition \ref{prop:super} and a simple rescaling argument (see for instance \cite{XiaZhou}), it is not hard to modify
 a generalized minimizing sequence into a classical minimizing sequence so that \eqref{eq:equalinf} holds. 
 
By \eqref{eq:equalinf}, in order to prove the existence of a generalized minimizer we  can consider a classical minimizing sequence $(E_n)_{n\ge 1}$ such that 
\[
 \lim_{n\to \infty} \E_\lambda(E_n)=\inf\lt\{ \widetilde{\E}_\lambda(\widetilde{E}) \ : \ \sum_i |E^i|=\omega_d\rt\}.
\]
We now follow relatively closely the proof of \cite[Theorem 4.9]{golnov}. We first notice that using the unit ball $B_1$ as competitor, 
we have $\lim_{n\to\infty}\E_\lambda(E_n)\le \E_\lambda(B_1)\le C(1+\lambda)$. For every $n\ge 1$, let $Q_{n,i}$ be a partition of $\R^d$ into disjoint cubes  of sidelength $2$ and such that 
\[
 m_{n,i}=|E_n\cap Q_{n,i}|
\]
is a decreasing sequence. By the relative  isoperimetric inequality we have 
\[
 \sum_i m_{n,i}^{\frac{d-1}{d}}\le C \sum_i P(E_n,Q_{n,i})=CP(E_n)\le C(1+\lambda).
\]
Since $\sum_i m_{n,i}=\omega_d$, we have for every $I\ge 1$, and every $i\ge I$, $m_i\le m_I\le \omega_d/I$ and thus
\[
\sum_{i\ge I} m_{n,i}=\sum_{i\ge I} m_{n,i}^{\frac{d-1}{d}} m_{n,i}^{\frac{1}{d}}\le C I^{-\frac{1}{d}}\sum_{i\ge I} m_{n,i}^{\frac{d-1}{d}}\le C(1+\lambda) I^{-\frac{1}{d}}.
\]
This proves uniform tightness of $m_{n,i}$ and thus up to extraction we may assume that for every $i$, $m_{n,i}\to m_i$ with $\sum_i m_i= \omega_d$. Let now $z_{n,i}\in Q_{n,i}$. Up to a further extraction we may assume that for every
$i,j$,  $|z_{n,i}-z_{n,j}|\to c_{ij}\in[0,\infty]$ and $E_n-z_{n,i}\to E^i$ in $L^1_{loc}(\R^d)$. We now introduce an equivalence class by saying that $i\sim j$ if $c_{ij}<\infty$
and denote by $[i]$ the equivalence class of $i$. Notice that if $i\sim j$, $E^i$ and $E^j$ coincide up to a translation. For every equivalence class $[i]$ let $m_{[i]}=\sum_{j\in [i]} m_j$ so that 
\[
 \sum_{[i]} m_{[i]}=\sum_i m_i=\omega_d.
\]
By the $L^1_{loc}$ convergence of $E_n-z_{n,i}$ to $E^i$ and the definition of the equivalence relation, we have for every $j\in [i]$, $|E^j|=m_{[i]}$. 
Up to a relabeling we may now assume that there is a unique element $E^i$ in each equivalence class. We have thus constructed a generalized set $\widetilde{E}=(E^i)_{i\ge 1}$ such that $\sum_i |E^i|=\omega_d$. We are left with the proof of 
\begin{equation}\label{eq:toprovegen}
 \widetilde{\E}_\lambda(\widetilde{E})\le \liminf_{n\to \infty} \E_\lambda(E_n)=\inf\lt\{ \widetilde{\E}_\lambda(\widetilde{E}) \ : \ \sum_i |E^i|=\omega_d\rt\}.
\end{equation}
To this aim let $I\in \N$. And let $z_{n,1}, \cdots, z_{n,I}$ be as before such that $E_n-z_{n,i}$ to $E^i$ and $|z_{n,i}-z_{n,j}|\to \infty$ if $i\neq j$. 
For every $R>0$, if $n$ is large enough, $\min_{i\neq j} |z_{n,i}-z_{n,j}|\ge 4R$. By the co-area formula, for every $n$ there is $R_n\in(R,2R)$ such that 
\[
 \sum_{i=1}^I \mathcal{H}^{d-1}(\partial B_{R_n}(z_{n,i})\cap E_n)\le \frac{C}{R}.
\]
We now define $E^{i,R_n}=(B_{R_n}(z_{n,i})\cap E_n) -z_{n,i}$ so that on the one hand,
\begin{equation}\label{bound:per}
 \sum_{i=1}^I P(E^{i,R_n})\le P(E_n)+ \frac{C}{R}
\end{equation}
and on the other hand by \eqref{repulsive},
\[
 \sum_{i=1}^I \cWp^p(E^{i,R_n})\le \cWp^p\lt(\cup_{i=1}^I B_{R_n}(z_{n,i})\cap E_n\rt)\le  \cWp^p(E_n).
\]
From the bound \eqref{bound:per}, we conclude that up to extraction, $E^{i,R_n}$ converges in $L^1$ to a set $E^{i,R}$ as $n\to \infty$. Moreover, from the $L^1_{loc}$ convergence of $E_n-z_{n,i}$ to $E^i$
it is not hard to see that also $E^{i,R}$ converges to $E^i$ in $L^1$ as $R\to \infty$. We thus conclude that by lower semi-continuity of the perimeter and \eqref{semicont} that 
\[
  \sum_{i=1}^I P(E^{i,R})+\lambda \lt(\sum_{i=1}^I \cWp^p(E^{i,R})\rt)^\alpha\le \liminf_{n\to \infty} \lt(P(E_n) +\lambda \cWpa{E_n}\rt) +\frac{C}{R}.  
\]
Letting then $R\to \infty$ and finally $I\to \infty$ we conclude the proof of \eqref{eq:toprovegen}.
\end{proof}
Before proceeding further let us study the scaling of the energy.
\begin{proposition}
 For every fixed $d\ge 2$, $p\ge 1$ and $\alpha>0$, there exists $C=C(d,p,\alpha)>0$ such that for every $\lambda>0$,
 \begin{equation}\label{eq:scalingener}
  \frac{1}{C}\lt(1+\lambda\rt)^{\frac{1}{1+\alpha p}}\le \inf_{|E|=\omega_d} \E_\lambda(E)\le C\lt(1+\lambda\rt)^{\frac{1}{1+\alpha p}}.
 \end{equation}
Moreover, if $\widetilde{E}=(E^i)_{i\ge 1}$ is a generalized minimizer, then 
\begin{equation}\label{eq:scalingP}
  \frac{1}{C}\lt(1+\lambda\rt)^{\frac{1}{1+\alpha p}}\le \sum_i P(E^i)\le C\lt(1+\lambda\rt)^{\frac{1}{1+\alpha p}}
\end{equation}
and 
\begin{equation}\label{eq:scalingW}
  \frac{1}{C}\lt(1+\lambda\rt)^{-\frac{ p}{1+\alpha p}}\le \sum_i \cWp^p(E^i)\le C\lt(1+\lambda\rt)^{-\frac{ p}{1+\alpha p}}.
\end{equation}
\end{proposition}
\begin{proof}
 Let us first consider the case $\lambda\le 1$. Using the ball $B_1$ as competitor and the isoperimetric inequality we have for every generalized minimizer $\widetilde{E}$,
 \[
  P(B_1)+\lambda \lt[\sum_i \cWp^p(E^i)\rt]^\alpha\le \sum_i P(E^i)+\lambda \lt[\sum_i \cWp^p(E^i)\rt]^\alpha\le P(B_1)+\lambda \cWpa{B_1}.
 \]
From this combined with the isoperimetric inequality we obtain \eqref{eq:scalingener} and \eqref{eq:scalingP} together with $\sum_i \cWp^p(E^i)\le \cWp(B_1)$. 
To obtain the first inequality in \eqref{eq:scalingW} we combine \eqref{eq:interpol} with $\sum_i P(E^i)\le C$.\\

Let now  $\lambda\ge 1$. We  consider the competitor made of $N$ balls $E^i$ of radius $r$ so that the constraint $\sum_i |E^i|=\omega_d$ translates into $Nr^d=1$. The energy of such a competitor is 
\[
 \widetilde{\E}_\lambda(\widetilde{E})= C\lt( r^{-1}+\lambda r^{p\alpha}\rt).
\]
Minimizing in $r$ by choosing $r=\lambda^{-\frac{1}{1+\alpha p}}$ (which is admissible since the corresponding $N$ is large) gives the upper bounds in \eqref{eq:scalingener}, \eqref{eq:scalingP} and \eqref{eq:scalingW}.
Using $\eqref{eq:interpol}$ we see that the upper bound in \eqref{eq:scalingP} gives the lower bound in \eqref{eq:scalingW} and vice-versa. These lower bounds then also imply the lower bound in \eqref{eq:scalingener}.
\end{proof}

\subsection{Quasi-minimality properties of generalized minimizers}
As in many similar variational problems, in order to prove a quasi-minimality property, it will be convenient to relax the volume constraint.
To this aim, for $\Lambda>0$ and $\widetilde{E}$ a generalized set we introduce the penalized  energy
\begin{equation}\label{def:superrelax}
 \widetilde{\E}_{\lambda,\Lambda}(\widetilde{E})=\sum_i P(E^i)+\lambda\lt[\sum_i \cWp^p(E^i)\rt]^\alpha +\Lambda \lt|\sum_i |E_i|-\omega_d\rt|.
\end{equation}
We start by proving that if $\Lambda$ is large enough, then every generalized minimizer is also an unconstrained minimizer of $\widetilde{\E}_{\lambda,\Lambda}$.
\begin{proposition}\label{prop:relax}
 There exists $C=C(d,p,\alpha)>0$ such that for every $\lambda>0$, if $\Lambda\ge C\lt(1+\lambda\rt)^{\frac{1}{1+\alpha p}}$ then every generalized minimizer of \eqref{def:gener} is also a minimizer of $\widetilde{\E}_{\lambda,\Lambda}$.
\end{proposition}
\begin{proof}
 Let $C_0$ to be fixed below and assume that $\Lambda\ge C_0\lt(1+\lambda\rt)^{\frac{1}{1+\alpha p}}$. By \eqref{eq:scalingener}, if there exists $\widetilde{E}$ such that 
\[
  \widetilde{\E}_{\lambda,\Lambda}(\widetilde{E})< \inf\lt\{ \widetilde{\E}_\lambda(\widetilde{E}') \ : \ \sum_i |(E')^i|=\omega_d\rt\}
\]
we must have 
\begin{equation}\label{eq:Lambda}
 \Lambda \lt|\sum_i |E^i|-\omega_d\rt|\le C\lt(1+\lambda\rt)^{\frac{1}{1+\alpha p}}.
\end{equation}
and $\sum_i |E^i|\neq \omega_d$. Let 
\[
 t =\omega_d^{\frac{1}{d}}\lt(\sum_i |E^i|\rt)^{-d}
\]
so that $t\widetilde{E}=(t E^i)_{i\ge 1}$ satisfies $\sum_i |t E^i|=\omega_d$. From \eqref{eq:Lambda}, we see that $t=1+\eps$ with
$|\eps|\le C\Lambda^{-1} \lt(1+\lambda\rt)^{\frac{1}{1+\alpha p}}$. By hypothesis we have 
\[
  \widetilde{\E}_{\lambda,\Lambda}(\widetilde{E})<  \widetilde{\E}_{\lambda}(t\widetilde{E})= (1+\eps)^{d-1}\sum_i P(E^i) +\lambda (1+\eps)^{(d+p)\alpha}\lt[\sum_i \cWp^p(E^i)\rt]^\alpha.
\]
By Taylor expansion we have for $\Lambda\ge C\lt(1+\lambda\rt)^{\frac{1}{1+\alpha p}}$,
\[
 \Lambda \eps< C\eps \lt(\sum_i P(E^i) +\lt[\sum_i \cWp^p(E^i)\rt]^\alpha\rt)\le C\eps\lt(1+\lambda\rt)^{\frac{1}{1+\alpha p}}.
\]
This gives the bound $\Lambda\le C\lt(1+\lambda\rt)^{\frac{1}{1+\alpha p}}$ which yields the conclusion provided $C_0>C$.
\end{proof}
Combining Lemma \ref{wassloc} together with Proposition \ref{prop:relax} we may now prove that generalized minimizers are $\Lambda-$minimizers of the perimeter.
\begin{proposition}\label{prop:Lambdamin}
 There exists $C=C(d,p,\alpha)>0$ such that if $\Lambda\ge C\lt(1+\lambda\rt)^{\frac{1+p}{1+\alpha p}}$, every generalized minimizer
 $\widetilde{E}=(E^i)_{i\ge 1}$ of $\widetilde{\E}_\lambda$ is a $\Lambda-$minimizer of the perimeter in the sense that for every $i\ge 1$ and every set $E\subset \R^d$,
 \begin{equation}\label{eq:Lambdamin}
  P(E^i)\le P(E)+\Lambda |E^i\Delta E|.
 \end{equation}

\end{proposition}
\begin{proof}
 Let $\Lambda_0= C\lt(1+\lambda\rt)^{\frac{1}{1+\alpha p}}$ be such that Proposition \ref{prop:relax} applies and let  $\widetilde{E}=(E^i)_{i\ge 1}$ be a generalized minimizer of $\widetilde{\E}_\lambda$. Without loss of generality,
 let us prove \eqref{eq:Lambdamin} for $E^1$.
 Using as competitor $E\times(E^i)_{i\ge 2}$ for $\widetilde{\E}_{\lambda,\Lambda_0}$ we find after simplification that 
 \begin{equation}\label{eq:testmin}
  P(E^1)\le P(E)+ \lambda \lt(\lt[ \cWp^p(E)+\sum_{i\ge 2} \cWp^p(E^i)\rt]^\alpha- \lt[\sum_i \cWp^p(E^i)\rt]^\alpha \rt) +\Lambda_0|E^1\Delta E|.
 \end{equation}
 Notice  that we can now  assume that 
 \begin{equation}\label{eq:assumeW} 
 \cWp^p(E)\ge \cWp^p(E^1)
 \end{equation}
 since otherwise we can already conclude that \eqref{eq:Lambdamin} holds.
 Moreover, \eqref{eq:scalingP} implies in particular that  $P(E^1)\le C(1+\lambda)^{\frac{1}{1+\alpha p}}$ so that we can assume that 
 \[
  |E^1\Delta E|\le C \Lambda^{-1} (1+\lambda)^{\frac{1}{1+\alpha p}}\le C (1+\lambda)^{-\frac{p}{1+\alpha p}},
 \]
 which in particular yields $|E^1|\le C$. From \eqref{eq:LipW}, this implies that we can work under the assumption
 \begin{equation}\label{assumpWE}
  \cWp^p(E)\le \cWp^p(E^1) + C  (1+\lambda)^{-\frac{p}{1+\alpha p}}\stackrel{\eqref{eq:scalingW}}{\le} C (1+\lambda)^{-\frac{p}{1+\alpha p}}.
 \end{equation}
Combining \eqref{eq:testmin}, \eqref{eq:LipW2} and \eqref{eq:LipW} we find

 \begin{align*}
  P(E^1)&\le P(E)+ \lambda\max\lt(\lt(\sum_i \cWp^p(E^i)\rt)^{\alpha-1}, \lt( \cWp^p(E)+\sum_{i\ge 2} \cWp^p(E^i)\rt)^{\alpha-1}\rt) |E^1\Delta E|\\
  &\qquad \qquad +\Lambda_0|E^1\Delta E|\\
  &\stackrel{\eqref{eq:assumeW}\&\eqref{assumpWE}\&\eqref{eq:scalingW}}{\le} P(E)+ C\lt(\lambda (1+\lambda)^{-\frac{p(\alpha-1)}{1+\alpha p}} +(1+\lambda)^{\frac{1}{1+\alpha p}}\rt) |E^1\Delta E|\\
  &\le P(E) +C (1+\lambda)^{\frac{1+p}{1+\alpha p}} |E^1\Delta E|.
 \end{align*}
 This  proves \eqref{eq:Lambdamin}. 
\end{proof}
As a direct corollary we obtain uniform density estimates for generalized minimizers (see \cite[Theorem 21.11]{Maggi}).
\begin{proposition}\label{prop:dens}
 There exists $C=C(d,p,\alpha)>0$ such that if   $ r<C\lt(1+\lambda\rt)^{-\frac{1+p}{1+\alpha p}}$, every generalized minimizer $\widetilde{E}=(E^i)_{i\ge1}$ of $\widetilde{\E}_\lambda$ satisfies for every $i$ and every $x\in E^i$
 \begin{equation}\label{eq:dens}
  |E^i\cap B (x,r)|\ge \frac{\omega_d}{4^d} r^d.
 \end{equation}
As a consequence, up to relabeling,  we have  $\widetilde{E}=(E^i)_{i=1}^I$ where for every $i$, $E^i$ are compact connected sets such that $\mathcal{H}^{d-1}(\partial E^i)=P(E^i)$. 
Moreover, there is a constant $C=C(d,p,\alpha)>0$ such that
\begin{equation}\label{eq:boundI}
   \sum_{i=1}^I \diam(E^i)\le C (1+\lambda)^{\frac{(d-1)(1+p)}{1+\alpha p}} \quad \textrm{ and } \quad \inf_i \diam(E^i)\ge C  (1+\lambda)^{-\frac{1+p}{1+\alpha p}}.
 \end{equation}
As a consequence $I\le C (1+\lambda)^{\frac{d(1+p)}{1+\alpha p}}$.
\end{proposition}
\begin{remark}\label{rem:reg}
 Let us notice that the regularity theory for $\Lambda-$minimizers of the perimeter gives us actually much more. 
 Denote   $\partial^*E$  the reduced boundary of $E$ (see \cite{Maggi}) and $\Sigma(E)=\partial E\backslash \partial^* E$. Then if $E$ is a $\Lambda-$minimizer of the perimeter, 
  $\partial^* E$ is  $C^{1,\gamma}$ for every $\gamma<1/2$ and $\Sigma(E)$ is empty if $d\le 7$, an at most finite union of points if $d=8$ and satisfies $\mathcal{H}^{s}(\Sigma(E))=0$ for every $s>d-8$ if $d\ge 9$.
  For classical or generalized minimizers of our energy we expect higher regularity to hold but this goes beyond the scope of this paper.
\end{remark}
\begin{proof}[Proof of Proposition \ref{prop:dens}]
 By Proposition \ref{prop:Lambdamin}, there exists $\Lambda=C\lt(1+\lambda\rt)^{\frac{1+p}{1+\alpha p}}$ such that every generalized minimizer $\widetilde{E}=(E^i)_{i\ge 1}$ is a $\Lambda-$minimizer of the perimeter. 
 By \cite[Theorem 21.11]{Maggi}, \eqref{eq:dens} holds as long as  $\Lambda r<1$. This proves the first part of the claim. We can further make the identification
 \[
  E^i=\{x\in \R^d \ : \ \liminf_{r\to 0} |E^i\cap B(x,r)|>0\}
 \]
 so that thanks to \eqref{eq:dens}, $E^i$ are closed sets with $\mathcal{H}^{d-1}(\partial E^i)=P(E^i)$. By \eqref{repulsive} we may further assume that each $E^i$ is connected. 
Fix now $r$ such that $\Lambda r=1/2$. By Vitali's covering Lemma, for every $i$ let $x_1,\cdots, x_{N_i}\in E^i$ be such that $E^i\subset \cup_{j=1}^{N_i} B(x_j,r)$ 
and $B(x_j,r/5)$ are pairwise disjoint. Using \eqref{eq:dens} we have $N_i\le C r^{-d} |E^i|$. Since $\diam(E^i)\le C  r N_i$ we have 
\[
 \sum_{i} \diam(E^i)\le C r^{-(d-1)}\le C (1+\lambda)^{\frac{(d-1)(1+p)}{1+\alpha p}}. 
\]
This proves the first part of \eqref{eq:boundI}. The second part follows from $\diam(E^i)\ge C |E^i|^{1/d}\ge C r$ which is a direct consequence of \eqref{eq:dens}.
\end{proof}

\subsection{Proof of Theorem \ref{exis_mini}}
We may now conclude the proof of Theorem \ref{exis_mini} and show the existence of (classical) minimizers for \eqref{def:minprob}.
\begin{proof}[Proof of Theorem \ref{exis_mini}]
  For every fixed $d\ge 2$, $p\ge 1$, $\alpha>0$ and $\lambda>0$, Proposition \ref{prop:existgenmin} gives the existence of a generalized minimizer $\widetilde{E}=(E^i)_{i\ge 1}$. 
  We thus have by \eqref{eq:equalinf},
  \begin{equation*}\label{eq:attained}
   \sum_i P(E^i)+\lambda\lt[\sum_i \cWp^p(E^i)\rt]^\alpha=\inf_{|E|=\omega_d} \E_\lambda(E).
  \end{equation*}
Thanks to Proposition \ref{prop:dens}, if $R=C (1+\lambda)^{\frac{(d-1)(1+p)}{1+\alpha p}}$ with $C>0$ large enough, then  $\widetilde{E}=(E^i)_{i=1}^I$ with $I\le R^{\frac{d}{d-1}}$ and for every $i\le I$, $E_i$ 
is a connected compact set with $\sum_{i=1}^I\diam(E^i)\le C R$. Let $(e_1,\cdots,e_d)$ be the canonical basis of $\R^d$ and define the set  
\[
 E=\cup_{i=1}^I (E^i+ R i e_1).
\]
By Proposition \ref{prop:super}, if $C$ is large enough, $\cWp^p(E)=\sum_i \cWp^p(E^i)$.
Since $E^i$ are pairwise disjoint we also have $P(E)=\sum_i P(E^i)$ (and $|E|=\sum_i |E^i|=\omega_d$) so that 
\[
 \E_\lambda(E)=  \sum_i P(E^i)+\lambda\lt[\sum_i \cWp^p(E^i)\rt]^\alpha\stackrel{\eqref{eq:equalinf}}{=}\inf_{|E|=\omega_d} \E_\lambda(E).
\]
Therefore $E$ is a minimizer of \eqref{def:minprob} and the proof is complete.
\end{proof}
\section{Minimality of the ball for small $\lambda$}
We now turn to Theorem \ref{ballmin} and prove that for small $\lambda$ the unique minimizers of \eqref{def:minprob} are balls. We first show that for $\lambda$ small enough, up to a translation, every minimizer of \eqref{def:minprob} is a small $C^{1,\gamma}$ perturbation of the ball $B_1$.
\begin{proposition}\label{prop:nearlyspherical}
 For every $d\ge 2$, $p\ge 1$, $\alpha>0$, $\gamma\in(0,1/2)$ and $\eps>0$, there exists $\lambda_0=\lambda_0(d,p,\alpha,\gamma, \eps)$ such that for every $\lambda\le \lambda_0$, 
 up to translation, every minimizer $E$ of \eqref{def:minprob} is nearly spherical in the sense that its barycenter is in $0$ and 
 there exists $f:\partial B_1\mapsto \R$ with $\|f\|_{C^{1,\gamma}}\le \eps$ such that 
 \[
  \partial E=\{(1+f(x)) x \ : \ x\in \partial B_1\}.
 \]

\end{proposition}
\begin{proof}
 The proof is quite classical and mostly rests on the (uniform in $\lambda$) $\Lambda-$minimizing property of $E$. Let $E_\lambda$ be a sequence of minimizers of \eqref{def:minprob}. For fixed $\gamma\in (0,1/2)$ we aim at proving 
 that up to translation $E_\lambda$ converges in $C^{1,\gamma}$ to $B_1$. We start by noting that using $B_1$ as a competitor together
 with the quantitative isoperimetric inequality we have up to translation,
 \begin{equation}\label{eq:estimDeltaW}
  |E_\lambda\Delta B_1|^2\le C\lt(P(E)-P(B_1)\rt)\le C\lambda \lt(\cWpa{B_1}-\cWpa{E_\lambda}\rt)\le C\lambda \cWpa{B_1}.
 \end{equation}
Therefore $E_\lambda$ converges in $L^1$ to $B_1$. It is now a classical fact that if a sequence of $\Lambda-$minimizers converges in $L^1$ to a smooth set then the whole sequence is actually smooth
(with the notation of Remark \ref{rem:reg}, $\Sigma(E_\lambda)=\emptyset$) and the convergence holds in $C^{1,\gamma}$ (see e.g. \cite[Lemma 3.6]{CicLeo}). As a consequence also the barycenter of $E_\lambda$ converges to $0$ and the proof is concluded. \
\end{proof}
We now recall that for nearly spherical sets, it was shown in \cite{fuglede} that there exists $C=C(d)>0$ such that 
\begin{equation}\label{eq:fugledeper}
 \int_{\partial B_1} f^2\le C \lt(P(E)-P(B_1)\rt).
\end{equation}
We postpone the proof of the following counterpart for $\cWp$ to the next section.
\begin{proposition}\label{prop:fugledeW}
 There exists $C=C(d,p,\alpha)>0$ such that for every  nearly spherical set $E$,
 \begin{equation}\label{eq:fugledeW}
  \cWpa{B_1}-\cWpa{E}\le C\int_{\partial B_1} f^2.
 \end{equation}
\end{proposition}
Taking this estimate for granted we may easily conclude the proof of Theorem \ref{ballmin}.
\begin{proof}[Proof of Theorem \ref{ballmin}]
 By Proposition \ref{prop:nearlyspherical}, if $\lambda$ is small enough then every minimizer $E$ of \eqref{def:minprob} is nearly spherical. Arguing as in \eqref{eq:estimDeltaW} we have 
 \[
  \int_{\partial B_1} f^2\stackrel{\eqref{eq:fugledeper}}{\le}C\lt(P(E)-P(B_1)\rt)\le C\lambda \lt(\cWpa{B_1}-\cWpa{E}\rt)\stackrel{\eqref{eq:fugledeW}}{\le} C\lambda \int_{\partial B_1} f^2,  
 \]
which implies that if $\lambda$ is small enough, $f=0$ and thus $E=B_1$.
\end{proof}
\subsection{Proof of Proposition \ref{prop:fugledeW}}
We start with a few simple facts about $\cWp(B_1)$. We let $A=B_{2^{1/d}}\backslash B_1$ be the annulus of volume $\omega_d$ around $B_1$.  With a slight abuse of notation, we will write $\phi(x)=\phi(|x|)$ if $\phi$ is a radial function.

\begin{lemma}\label{lem:Wasball}
 We have the following properties:
 \begin{itemize}
  \item[(i)] $A$ is the minimizer of \eqref{def:Wp} for $B_1$, i.e.  $\cWp(B_1)=W_p(B_1,A)$;
  \item[(ii)] The map 
  \begin{equation}\label{def:T}
   T(x)=\lt(1+|x|^d\rt)^{\frac{1}{d}} \frac{x}{|x|}
  \end{equation}
is an optimal transport map (the unique one if $p>1$) between $B_1$ and $A$. Moreover, the corresponding Kantorovich potentials $(\phi,\psi)$ are radially symmetric and $r\mapsto \psi(r)$ is increasing. Finally, $(\phi,\psi)$ are locally Lipschitz continuous.
 \end{itemize}

\end{lemma}
\begin{proof}
We start with $(i)$. By Proposition \ref{prop:Wb} let $F$ be the unique minimizer of \eqref{def:Wp} for $B_1$ so that $\cWp(B_1)=W_p(B_1,F)$.  If $R$ is  any rotation of $\R^d$, since 
$W_p(R(B_1), R(F))= W_p(B_1, F)$, we see that $R(F)$ is also a minimizer of \eqref{def:Wp} for $B_1$. By uniqueness we have $F=R(F)$ and thus $F$ is radially symmetric. 
	 It is then immediate to check that among radially symmetric sets the minimizer is indeed $A$.
	
	\bigskip
	
	We now turn to  $(ii)$. We start by noting that $T$ defined in \eqref{def:T} is the unique radially symmetric map (in the sense that $T(x)=f(|x|) \frac{x}{|x|}$) which solves $\det \nabla T=1$ and $f(0)=1$. 
	Let us argue that $T$ is $c-$cyclically monotone for the cost $c(x,y)=|x-y|^p$ and thus an optimal transport map between any bounded radially symmetric set $E$ and $T(E)$ (see \cite[Definition 2.33 \& Remark 2.39]{Viltop}).
	This follows from the fact that $f(r)=(1+r^d)^{1/d}$ is monotone on $\R^+$ and thus also $c-$cyclically monotone on $\R^+$ (as these two notions coincide for convex costs in dimension one)
	so that for every $x_1,\cdots, x_I$, using the convention $x_0=x_I$
	\begin{multline*}
	 \sum_{i=1}^{I} |T(x_i)-x_i|^p=\sum_{i=1}^I |f(|x_i|)-|x_i||^p
	 \le \sum_{i=1}^I |f(|x_{i-1}|)-|x_{i}||^p\\
	 \le \sum_{i=1}^I |f(|x_{i-1}|) \frac{x_{i-1}}{|x_{i-1}|}-x_{i}|^p
	 =\sum_{i=1}^I |T(x_{i-1})-x_i|^p.
	\end{multline*}
Notice that the inverse map $T^{-1}: B_1^c\mapsto \R^d$ is given by
\[
 T^{-1}(y)=\bigg(|y|^d -1 \bigg)^{\frac{1}{d}}\frac{y}{|y|}.
\]

	We now argue a bit differently for  $p>1$ and $p=1$ regarding the Kantorovich potentials. Let us start with the easier case $p=1$. Denoting $\phi(x)=-|x|$ we have that $\phi$ is $1-$Lipschitz, radially symmetric  
	and decreasing (and thus $\psi=-\phi$ is radially symmetric and increasing) and satisfies for $x\in \R^d$
	\begin{equation}\label{eq:phi1}
	 \phi(x)-\phi(T(x))=|T(x)|-|x|=|T(x)-x|
	\end{equation}
so that $(\phi,-\phi)$ is indeed a couple of Kantorovich potentials. As a side note,  it is easily seen from \eqref{eq:phi1} that on the one hand, up to a constant $\phi$ is the unique Kantorovich potential and on
the other hand that every optimal transport map must be radially symmetric (there is however no uniqueness of the optimal transport map).
Note also that the validity of  \eqref{eq:phi1} gives an alternative proof of the optimality of $T$ when $p=1$.\\
For $p>1$, we first argue that $\phi$ is radially symmetric and decreasing. For this we use that by \cite[Theorem~1.17]{otam}, if we let $h(z)=|z|^p$, then the unique Kantorovich potential $\phi$ is given by
\[
 \nabla \phi(x)=\nabla h(x-T(x))=-p\lt( (1+|x|^d)^{\frac{1}{d}}-|x|\rt)^{p-1} \frac{x}{|x|}=\phi'(|x|)\frac{x}{|x|}
\]
with $\phi'\le 0$. Now since $\phi$ and $\psi$ are $c-$conjugate, we have 
\begin{equation}\label{eq:cconj}
 \psi(y)=\inf_x\lt[ |x-y|^p-\phi(x)\rt]
\end{equation}
from which we deduce that also $\psi$ is radially symmetric. Arguing exactly as for $\phi$ but with $T$ replaced by $T^{-1}$ we see that $\psi$ is increasing on $B_1^c$.
In order to conclude that $\psi$ is in fact increasing on $\R^d$ we will prove that for $y\in B_1$,
\begin{equation}\label{eq:toprovepsi}
 \psi(y)=|y|^p-\phi(0)
\end{equation}
or in other words that \eqref{eq:cconj} is attained at $x=0$. We first point out that \eqref{eq:toprovepsi} holds for $|y|=1$ since $T^{-1}(y)=0$ and thus by definition of Kantorovich potentials
\[
 \phi(0)+\phi(y)=|y|^p.
\]
We also observe that since $\phi$ is decreasing, for every $y\in \R^d$ the optimal $x$ in \eqref{eq:cconj} must satisfy $|x|\le |y|$ (and $x= |x| y/|y|$). Fix now $y\in B_1$ and let $x$ be such that 
\[
 \psi(y)=|y-x|^p-\phi(x).
\]
Let $\bar y= y/|y|\in \partial B_1$. Using $x$ as a competitor in \eqref{eq:cconj}  for $\bar y$ we have 
\[
 \psi(\bar y)=1-\phi(0)\le (1-|x|)^p-\phi(x).
\]
Using now $0$ as competitor in \eqref{eq:cconj} for $y$ we also have 
\[
 (|y|-|x|)^p-\phi(x)\le |y|^p-\phi(0)
\]
so that 
\[
1-(1-|x|)^p\le \phi(0)-\phi(x)\le |y|^p-  (|y|-|x|)^p.
\]
However the function $t\to t^p-(t-|x|)^p$ is increasing in $[|x|,\infty)$ so that we reach a contradiction unless $x=0$.\\
To conclude, the local Lipschitz continuity of $(\phi,\psi)$ is standard, see \cite[Proposition 2.43]{Viltop}.
\end{proof}
In order to prove \eqref{eq:fugledeW} we will need the following simple result.
\begin{lemma}\label{lem:technical}
 Let $\psi$ be a radially symmetric and increasing function and let $E\subset B_{2^{1/d}}$ with $|E|=\omega_d$. Then
 \begin{equation}\label{eq:technical}
  \inf_F\lt\{ \int_F \psi \ : \ |F\cap E|=0 \textrm{ and } |F|=\omega_d\rt\}=\int_{B_{2^{1/d}}\backslash E} \psi.
 \end{equation}

\end{lemma}
\begin{proof}
 We first show that for any $r>0$,
	\begin{equation}\label{eq:firstmin}
	\min_{|E|=|B_r|} \int_E \psi=\int_{B_r} \psi.
	\end{equation}
	For $E$ with $|E|=|B_r|$, we write
	\[
	\int_E \psi - \int_{B_r} \psi = \int_{E \setminus B_r} \psi - \int_{B_r \setminus E} \psi.
	\]
	\noindent Since $\psi$ is radially increasing we have 
	\[
	\inf_{E \setminus B_r} \psi \geq \psi(r) \geq \sup_{B_r \setminus E} \psi.
	\]
	\noindent Using $|E\backslash B_r|=|B_r\backslash E|$, we find
	\[
	\int_E \psi - \int_{B_r} \psi \geq 0
	\]
	and thus  \eqref{eq:firstmin} holds.\\
	\noindent Now if  $E\subset B_{2^{1/d}}$  with $|E|=\omega_d$, for every set  $F$ with  $|F\cap E|=0 $ and $|F|=|E|=\omega_d$, we have $|E\cup F|=|B_{2^{1/d}}|$ and thus
	\[
	\int_F \psi = \int_{F \cup E} \psi - \int_E \psi \stackrel{\eqref{eq:firstmin}}{\geq} \int_{B_{2^{1/d}}} \psi - \int_E \psi = \int_{B_{2^{1/d}} \setminus E} \psi,
	\]
	\noindent which is the desired conclusion.
\end{proof}
We may now prove Proposition \ref{prop:fugledeW}.
\begin{proof}[Proof of Proposition \ref{prop:fugledeW}]
 We may assume that $\cWp(B_1)\ge \cWp(E)$ since otherwise there is nothing to prove. Using \eqref{eq:LipW2} we see that it is enough to prove the estimate for $\alpha=1$, that is
 \begin{equation}\label{eq:toprovefuglede}
  \cWp^p(B_1)-\cWp^p(E)\le C\int_{\partial B_1} f^2.
 \end{equation}
Let $(\phi,\psi)$ be the Kantorovich potentials associated with $W_p(B_1,A)$ and recall that by Lemma \ref{lem:Wasball}, $\psi$ is radially symmetric and increasing. Since $E$ is nearly spherical we have $E\subset B_{2^{1/d}}$.
For every admissible competitor $F$ for $\cWp(E)$ we have by duality
\[
 W_p^p(E,F)\ge \int_E \phi +\int_F \psi.
\]
Taking the infimum over $F$ we get 
\[
 \cWp^p(E)\ge \int_E \phi +\inf_F\lt\{ \int_F \psi \ : \ |F\cap E|=0 \textrm{ and } |F|=\omega_d\rt\}\stackrel{\eqref{eq:technical}}{\ge} \int_E \phi+ 
 \int_{B_{2^{1/d}}\backslash E} \psi.
\]
Therefore, 
\begin{align*}
 \cWp^p(B_1)-\cWp^p(E)&\le \int_{B_1} \phi +\int_{A} \psi - \int_E \phi-\int_{B_{2^{1/d}}\backslash E} \psi\\
 &= \int_{B_1} (\phi-\psi) -\int_{E} (\phi-\psi)\\
 &=\int_{B_1\backslash E}(\phi-\psi) -\int_{E\backslash B_1} (\phi-\psi).
\end{align*}
We may now argue as in \cite[Proposition 6.2]{KnMu}. We let $c=\phi(1)-\psi(1)$ and use that $\phi$ and $\psi$ are Lipschitz continuous in a neighborhood of $\partial B_1$ to infer 
\begin{align*}
 \int_{B_1\backslash E}(\phi-\psi) -\int_{E\backslash B_1} (\phi-\psi)&=\int_{B_1\backslash E}[(\phi-\psi)-c] -\int_{E\backslash B_1} [(\phi-\psi)-c]\\
 &\le C \int_{B_1\Delta E} |1-|x||\\
 &\le C \int_{\partial B_1} \int_0^{f(x)} t dt d\mathcal{H}^{d-1}(x)\\
 &\le C\int_{\partial B_1} f^2.
\end{align*}
This concludes the proof of \eqref{eq:toprovefuglede}.
\end{proof}
\bibliographystyle{acm}
\bibliography{biblio_cangol}

\begin{thebibliography}{10}

\bibitem{AcFuMo}
{\sc Acerbi, E., Fusco, N., and Morini, M.}
\newblock Minimality via second variation for a nonlocal isoperimetric problem.
\newblock {\em Comm. Math. Phys. 322}, 2 (2013), 515--557.

\bibitem{BuCaLa}
{\sc Buttazzo, G., Carlier, G., and Laborde, M.}
\newblock On the {W}asserstein distance between mutually singular measures.
\newblock {\em Adv. Calc. Var. 13}, 2 (2020), 141--154.

\bibitem{CaFuPra}
{\sc Carazzato, D., Fusco, N., and Pratelli, A.}
\newblock Minimality of balls in the small volume regime for a general {G}amow
  type functional, 2020.

\bibitem{ChMuTo}
{\sc Choksi, R., Muratov, C.~B., and Topaloglu, I.}
\newblock An old problem resurfaces nonlocally: {G}amow's liquid drops inspire
  today's research and applications.
\newblock {\em Notices Amer. Math. Soc. 64}, 11 (2017), 1275--1283.

\bibitem{CicLeo}
{\sc Cicalese, M., and Leonardi, G.~P.}
\newblock A selection principle for the sharp quantitative isoperimetric
  inequality.
\newblock {\em Arch. Ration. Mech. Anal. 206}, 2 (2012), 617--643.

\bibitem{DePMSV}
{\sc De~Philippis, G., M\'esz\`aros, A.~R., Santambrogio, F., and Velichkov,
  B.}
\newblock {BV} estimates in optimal transportation and applications.
\newblock {\em Archive for Rational Mechanics and Analysis 219}, 2 (Sep 2015),
  829–860.

\bibitem{F2M3}
{\sc Figalli, A., Fusco, N., Maggi, F., Millot, V., and Morini, M.}
\newblock Isoperimetry and stability properties of balls with respect to
  nonlocal energies.
\newblock {\em Comm. Math. Phys. 336}, 1 (2015), 441--507.

\bibitem{FraLieb}
{\sc Frank, R.~L., and Lieb, E.~H.}
\newblock A compactness lemma and its application to the existence of
  minimizers for the liquid drop model.
\newblock {\em SIAM J. Math. Anal. 47}, 6 (2015), 4436--4450.

\bibitem{FranNam}
{\sc Frank, R.~L., and Nam, P.~T.}
\newblock Existence and nonexistence in the liquid drop model, 2021.

\bibitem{fuglede}
{\sc Fuglede, B.}
\newblock Stability in the isoperimetric problem for convex or nearly spherical
  domains in {$\mathbb{R}^n$}.
\newblock {\em Transactions of the American Mathematical Society 314\/} (08
  1989).

\bibitem{golnov}
{\sc Goldman, M., and Novaga, M.}
\newblock Volume-constrained minimizers for the prescribed curvature problem in
  periodic media.
\newblock {\em Calc. Var. Partial Differential Equations 44}, 3-4 (2012),
  297--318.

\bibitem{GolNovRuf}
{\sc Goldman, M., Novaga, M., and Ruffini, B.}
\newblock Existence and stability for a non-local isoperimetric model of
  charged liquid drops.
\newblock {\em Arch. Ration. Mech. Anal. 217}, 1 (2015), 1--36.

\bibitem{KnMu}
{\sc Kn\"upfer, H., and Muratov, C.~B.}
\newblock On an isoperimetric problem with a competing nonlocal term {II}:
  {T}he general case.
\newblock {\em Comm. Pure Appl. Math. 67}, 12 (2014), 1974--1994.

\bibitem{KnMuNov}
{\sc Knupfer, H., Muratov, C.~B., and Novaga, M.}
\newblock Low density phases in a uniformly charged liquid.
\newblock {\em Communications in Mathematical Physics 345}, 1 (Jul 2016),
  141--183.

\bibitem{Maggi}
{\sc Maggi, F.}
\newblock {\em Sets of Finite Perimeter and Geometric Variational Problems: An
  Introduction to Geometric Measure Theory}.
\newblock Cambridge Studies in Advanced Mathematics. Cambridge University
  Press, 2012.

\bibitem{MuVes}
{\sc Mukoseeva, E., and Vescovo, G.}
\newblock Minimality of the ball for a model of charged liquid droplets, 2019.

\bibitem{NoToVenk}
{\sc {Novack}, M., {Topaloglu}, I., and {Venkatraman}, R.}
\newblock {Least Wasserstein distance between disjoint shapes with perimeter
  regularization}.
\newblock {\em arXiv e-prints\/} (Aug. 2021), arXiv:2108.04390.

\bibitem{novprat}
{\sc Novaga, M., and Pratelli, A.}
\newblock Minimisers of a general {R}iesz-type problem.
\newblock {\em Nonlinear Anal. 209\/} (2021), Paper No. 112346, 27.

\bibitem{Pegon}
{\sc Pegon, M.}
\newblock Large mass minimizers for isoperimetric problems with integrable
  nonlocal potentials.
\newblock {\em Nonlinear Anal. 211\/} (2021), Paper No. 112395, 48.

\bibitem{PeRo}
{\sc Peletier, M.~A., and R\"{o}ger, M.}
\newblock Partial localization, lipid bilayers, and the elastica functional.
\newblock {\em Arch. Ration. Mech. Anal. 193}, 3 (2009), 475--537.

\bibitem{Rigot}
{\sc Rigot, S.}
\newblock Ensembles quasi-minimaux avec contrainte de volume et
  rectifiabilit\'{e} uniforme.
\newblock {\em M\'{e}m. Soc. Math. Fr. (N.S.)}, 82 (2000), vi+104.

\bibitem{otam}
{\sc Santambrogio, F.}
\newblock {\em Optimal transport for applied mathematicians}, vol.~87 of {\em
  Progress in Nonlinear Differential Equations and their Applications}.
\newblock Birkh\"{a}user/Springer, Cham, 2015.
\newblock Calculus of variations, PDEs, and modeling.

\bibitem{Viltop}
{\sc Villani, C.}
\newblock {\em Topics in optimal transportation}, vol.~58 of {\em Graduate
  Studies in Mathematics}.
\newblock American Mathematical Society, Providence, RI, 2003.

\bibitem{XiaZhou}
{\sc Xia, Q., and Zhou, B.}
\newblock The existence of minimizers for an isoperimetric problem with
  {W}asserstein penalty term in unbounded domains:.
\newblock {\em Advances in Calculus of Variations\/} (2021),
  000010151520200083.

\end{thebibliography}

\end{document}